\documentclass[11pt]{article}
\usepackage{geometry,amsmath,amsthm,amssymb,latexsym}

\usepackage[svgnames]{xcolor}
\usepackage[colorlinks=true,linkcolor=MidnightBlue,citecolor=MidnightBlue,urlcolor=MidnightBlue,pdfpagelabels=false]{hyperref}

\usepackage{thmtools}
\theoremstyle{plain}
  \declaretheorem[numberwithin=section]{theorem}
  \declaretheorem[numberlike=theorem]{corollary}
  \declaretheorem[numberlike=theorem]{proposition}

  \declaretheorem[numberlike=theorem]{question}
\theoremstyle{definition}
  
  \declaretheorem[numberlike=theorem]{example}
  \declaretheorem[numberlike=theorem]{remark}

\newcommand{\tmverbatim}[1]{\text{{\ttfamily{#1}}}}

\begin{document}

\title{
  Partial Lucas-type congruences

  {\small{(Dedicated to George Andrews and Bruce~Berndt for their 85\textsuperscript{th}~birthdays)}}
}

\author{Armin Straub\thanks{\textit{Email:} \texttt{straub@southalabama.edu}}\\
Department of Mathematics and Statistics\\
University of South Alabama}

\date{November 2, 2025}

\maketitle

\begin{abstract}
  In their study of a binomial sum related to Wolstenholme's theorem,
  Chamberland and Dilcher prove that the corresponding sequence modulo primes
  $p$ satisfies congruences that are analogous to Lucas' theorem for the
  binomial coefficients with the notable twist that there is a restriction on
  the $p$-adic digits. We prove a general result that shows that similar
  partial Lucas congruences are satisfied by all sequences representable as
  the constant terms of the powers of a multivariate Laurent polynomial.
\end{abstract}

\section{Introduction}

Chamberland and Dilcher \cite{cd-binomial06} study divisibility properties
of the sequences $(u_{a, b}^{\varepsilon} (n))_{n \geq 0}$ defined by the
binomial sums
\begin{equation}
  u_{a, b}^{\varepsilon} (n) = \sum_{k = 0}^n (- 1)^{\varepsilon k}
  \binom{n}{k}^a \binom{2 n}{k}^b \label{eq:uabe:def}
\end{equation}
and connect these sums to Wolstenholme's theorem. In the follow-up study
\cite{cd-binomial09}, they further investigate the special case
$(\varepsilon, a, b) = (1, 1, 1)$
\begin{equation}
  u (n) = \sum_{k = 0}^n (- 1)^k \binom{n}{k} \binom{2 n}{k}, \label{eq:u:def}
\end{equation}
which is sequence \tmverbatim{A234839} in the OEIS \cite{oeis}. One of the
interesting arithmetic properties that Chamberland and Dilcher
\cite{cd-binomial09} prove is the following set of congruences which, as we
recall below, can be viewed as partial Lucas congruences.

\begin{theorem}[{\cite[Theorem~2.2]{cd-binomial09}}]
  \label{thm:u:lucas}Let $p \geq 3$ be a prime. Then
  \begin{equation*}
    u (p n + k) \equiv u (n) u (k) \pmod{p}
  \end{equation*}
  for all integers $n, k \geq 0$ with $k \leq (p - 1) / 2$.
\end{theorem}

Let $p$, here and throughout, denote a prime, while $n$ and $k$ are used to
denote nonnegative integers. It follows immediately, as observed in
\cite[Corollary~2.1]{cd-binomial09}, that, if $n = n_0 + n_1 p + \cdots +
n_r p^r$ is the $p$-adic expansion of $n$, then
\begin{equation}
  u (n) \equiv u (n_0) u (n_1) \cdots u (n_r) \pmod{p},
  \label{eq:lucas:u}
\end{equation}
provided that $0 \leq n_j < p / 2$ for all $j = 0, 1, \ldots, r$. These
congruences are reminiscent of the classical congruences
\begin{equation}
  \binom{n}{k} \equiv \binom{n_0}{k_0} \binom{n_1}{k_1} \cdots
  \binom{n_r}{k_r} \pmod{p} \label{eq:binom:lucas}
\end{equation}
for the binomial coefficients due to Lucas \cite{lucas78} where $n = n_0 +
n_1 p + \cdots + n_r p^r$ and $k = k_0 + k_1 p + \cdots + k_r p^r$ with $0
\leq n_j, k_j < p$.

For the sequence $(u (n))_{n \geq 0}$ from \eqref{eq:u:def}, the
congruences \eqref{eq:lucas:u} are restricted to digits $n_j$ less than $p /
2$. On the other hand, Gessel showed \cite[Theorem 1]{gessel-super} that the
Ap\'ery numbers
\begin{equation}
  A (n) = \sum_{k = 0}^n \binom{n}{k}^2 \binom{n + k}{k}^2 \label{eq:apery3}
\end{equation}
in place of $u (n)$ satisfy the congruences \eqref{eq:lucas:u} for all digits
$0 \leq n_j < p$ (see also Example~\ref{eg:apery3}). Accordingly, we say
that the Ap\'ery numbers satisfy \emph{Lucas congruences} modulo all
primes, and the same has since been shown to be true for many other (families
of) sequences \cite{mcintosh-lucas}, \cite{granville-bin97},
\cite{sd-laurent09}, \cite{ry-diag13}, \cite{ms-lucascongruences},
\cite{delaygue-apery}, \cite{abd-lucas}, \cite{gorodetsky-ct},
\cite{hs-lucas-x}. A historical survey of Lucas congruences can be found in
\cite{mestrovic-lucas}. On the other hand, the sequence $(u (n))_{n
\geq 0}$ satisfies the Lucas congruences only partially.

Chamberland and Dilcher \cite{cd-binomial09} proved
Theorem~\ref{thm:u:lucas} directly from the binomial sum \eqref{eq:u:def}
using the Lucas congruences \eqref{eq:binom:lucas} for the binomial
coefficients. We show that Theorem~\ref{thm:u:lucas}, as well as similar
results, can be uniformly proved for a large class of sequences representable
as \emph{constant terms}. Here, we say that a sequence $(c (n))_{n \geq
0}$ is a constant term if it can be expressed as
\begin{equation*}
  c (n) = \operatorname{ct} [P (\boldsymbol{x})^n Q (\boldsymbol{x})]
\end{equation*}
for Laurent polynomials $P (\boldsymbol{x}), Q (\boldsymbol{x}) \in \mathbb{Z}
[\boldsymbol{x}^{\pm 1}]$ in several variables $\boldsymbol{x}= (x_1, x_2, \ldots,
x_d)$ with integer coefficients. Throughout, we use the notation $\operatorname{ct} [F
(\boldsymbol{x})]$ for the constant term of a Laurent polynomial $F
(\boldsymbol{x})$ in the variables $\boldsymbol{x}$. As further detailed in
Section~\ref{sec:lucas:partial}, we denote with $\operatorname{Newt} (P)$ the Newton
polytope of $P$. Our main result shows that all constant term sequences with
$Q = 1$ satisfy partial Lucas congruences.

\begin{theorem}
  \label{thm:ct:lucas:partial:intro}Let $P (\boldsymbol{x}) \in \mathbb{Z}
  [\boldsymbol{x}^{\pm 1}]$ and choose $M \geq 1$ large enough so that
  $\frac{1}{M} \operatorname{Newt} (P)$ contains no integral (relative) interior point
  besides the origin. Then, for all primes $p$, the sequence defined by $A (n)
  = \operatorname{ct} [P (\boldsymbol{x})^n]$ satisfies
  \begin{equation*}
    A (p n + k) \equiv A (n) A (k) \pmod{p}
  \end{equation*}
  for all integers $n, k \geq 0$ with $k < \frac{p}{M}$.
\end{theorem}

We note that the special case $M = 1$ of
Theorem~\ref{thm:ct:lucas:partial:intro}, which proves that certain constant
terms satisfy the (full) Lucas congruences modulo all primes, also follows
from a result of Samol and van Straten \cite[Theorem~4.3]{sd-laurent09} (see
also \cite{mv-laurent13} and \cite{hs-lucas-x}). We prove
Theorem~\ref{thm:ct:lucas:partial:intro} in Section~\ref{sec:lucas:partial}
and speculate about a potential converse in Question~\ref{q:ct:lucas}. As a
first example, we observe here that it implies Theorem~\ref{thm:u:lucas} as a
special case.

\begin{example}
  \label{eg:u}We will show below that the sequence $(u (n))_{n \geq 0}$
  defined in \eqref{eq:u:def} has the constant term representation
  \begin{equation}
    u (n) = \operatorname{ct} \left[ \left((1 + x) \left(x - \frac{1}{x} \right)
    \right)^n \right] . \label{eq:u:ct}
  \end{equation}
  Clearly, the Newton polytope of $P (x) = (1 + x) \left(x - \frac{1}{x}
  \right) = x^2 + x - 1 - \frac{1}{x}$ is the interval $\operatorname{Newt} (P) = [-
  1, 2]$. In particular, $\frac{1}{2} \operatorname{Newt} (P) = \left[ - \frac{1}{2},
  1 \right]$ contains no integral interior point besides $0$. Therefore,
  Theorem~\ref{thm:ct:lucas:partial:intro} applies with $M = 2$ to show that
  Theorem~\ref{thm:u:lucas} indeed holds for all primes (note that
  Theorem~\ref{thm:u:lucas} holds rather vacuously for $p = 2$).
  
  In principle, once found, a constant term representation such as
  \eqref{eq:u:ct} can be algorithmically proven using creative telescoping
  \cite{koutschan-phd} (see, for instance, \cite{gorodetsky-ct} for
  worked-out examples). In general, however, it is a difficult problem to find
  constant term representations (or even to decide whether such a
  representation exists; see \cite{bsy-constantterms}). Yet, certain
  binomial sums can be systematically translated into constant term
  representations. This is explained by Rowland and Zeilberger in
  \cite{rz-cong} who attribute the approach to Egorychev. In the present
  case we readily obtain
  \begin{eqnarray*}
    u (n) & = & \sum_{k = 0}^n (- 1)^k \binom{n}{k} \binom{2 n}{k}\\
    & = & \operatorname{ct} \left[ \sum_{k = 0}^n (- 1)^k \binom{n}{k} \frac{(1 +
    x)^{2 n}}{x^k} \right]\\
    & = & \operatorname{ct} \left[ (1 + x)^{2 n} \left(1 - \frac{1}{x} \right)^n
    \right],
  \end{eqnarray*}
  which is equivalent to the claimed representation.
\end{example}

We illustrate the versatility of Theorem~\ref{thm:ct:lucas:partial:intro} by
showing that Theorem~\ref{thm:u:lucas} for the sequence defined by $u (n) =
u_{1, 1}^1 (n)$ actually continues to hold for the more general sequence
$(u_{a, b}^{\varepsilon} (n))_{n \geq 0}$ defined in \eqref{eq:uabe:def}.

\begin{corollary}
  For any prime $p$ and any integers $\varepsilon$, $a \geq 1$, $b
  \geq 0$ we have
  \begin{equation*}
    u_{a, b}^{\varepsilon} (p n + k) \equiv u_{a, b}^{\varepsilon} (n) u_{a,
     b}^{\varepsilon} (k) \pmod{p}
  \end{equation*}
  for all integers $n, k \geq 0$ with $k < p / 2$.
\end{corollary}

\begin{proof}
  We can readily express $u_{a, b}^{\varepsilon} (n)$ as a constant term as
  follows:
  \begin{eqnarray}
    u_{a, b}^{\varepsilon} (n) & = & \sum_{k = 0}^n (- 1)^{\varepsilon k}
    \binom{n}{k}^a \binom{2 n}{k}^b \nonumber\\
    & = & \operatorname{ct} \left[ \sum_{k = 0}^n (- 1)^{\varepsilon k} \binom{n}{k}
    \prod_{i = 1}^{a - 1} \frac{(1 + x_i)^n}{x_i^k} \prod_{j = a}^{a + b - 1}
    \frac{(1 + x_j)^{2 n}}{x_j^k} \right] = \operatorname{ct} [P (\boldsymbol{x})^n] 
    \label{eq:uabe:ct}
  \end{eqnarray}
  with the Laurent polynomial
  \begin{equation*}
    P (\boldsymbol{x}) = \left(1 + \frac{(- 1)^{\varepsilon}}{x_1 \cdots x_{a
     + b - 1}} \right) \prod_{i = 1}^{a - 1} (1 + x_i) \prod_{j = a}^{a + b -
     1} (1 + x_j)^2 .
  \end{equation*}
  Note that $\operatorname{Newt} (P) \subseteq [- 1, 2]^{a + b - 1}$ so that we can
  apply Theorem~\ref{thm:ct:lucas:partial:intro} with $M = 2$ to arrive at the
  claimed congruences.
\end{proof}

If desired, further generalizations beyond the sequence $(u_{a,
b}^{\varepsilon} (n))_{n \geq 0}$ can be given along the same lines (for
instance, one can replace $(- 1)^{\varepsilon}$ by any integer $r$ and one can
insert additional powers of suitable binomial coefficients such as $\binom{n +
k}{k}$ into the summation).

\section{Partial Lucas congruences}\label{sec:lucas:partial}

When working with several variables $\boldsymbol{x}= (x_1, x_2, \ldots, x_d)$,
we use common shorthand notation to write, for instance,
$\boldsymbol{x}^{\boldsymbol{v}} = x_1^{v_1} x_2^{v_2} \cdots x_d^{v_d}$ where the
exponents are specified in the vector $\boldsymbol{v}= (v_1, v_2, \ldots, v_d)$.
Given a Laurent polynomial
\begin{equation*}
  P (\boldsymbol{x}) = \sum_{\boldsymbol{v} \in \mathbb{Z}^d} c_{\boldsymbol{v}}
   \boldsymbol{x}^{\boldsymbol{v}} \in R [\boldsymbol{x}^{\pm 1}]
\end{equation*}
over some ring $R$ of characteristic $0$, its support is
\begin{equation*}
  \operatorname{supp} (P) = \left\{ \boldsymbol{v} \in \mathbb{Z}^d \, :
   \, c_{\boldsymbol{v}} \neq 0 \right\} .
\end{equation*}
For subsets $S, T$ of a vector space and scalars $\lambda$, we use the typical
notations $\lambda S = \left\{ \lambda \boldsymbol{v} \, :
\, \boldsymbol{v} \in S \right\}$ and $S + T = \left\{
\boldsymbol{v}+\boldsymbol{w} \, : \, \boldsymbol{v} \in S,
\boldsymbol{w} \in T \right\}$, the latter being the Minkowski sum. For
instance, we note that $\operatorname{supp} (P Q) \subseteq \operatorname{supp} (P) +
\operatorname{supp} (Q)$ for Laurent polynomials $P, Q$.

Some care has to be applied since $2 S \subseteq S + S$ while, in general, $2
S \neq S + S$. Such issues disappear if we work with convex sets. Indeed, for
all convex sets $S$ we do have $2 S = S + S$ (to see this, take $\boldsymbol{v},
\boldsymbol{w} \in S$ and note that $\boldsymbol{v}+\boldsymbol{w}= 2 \left(\frac{1}{2} \boldsymbol{v}+ \frac{1}{2} \boldsymbol{w} \right) \in 2 S$ because
convexity of $S$ implies that $\frac{1}{2} \boldsymbol{v}+ \frac{1}{2}
\boldsymbol{w} \in S$). For further background, we refer to
\cite{schneider-convex}. The Newton polytope of $P$, denoted as $\operatorname{Newt}
(P)$, is the convex hull of $\operatorname{supp} (P)$. In other words,
\begin{equation}
  \operatorname{Newt} (P) = \left\{ \sum_{\boldsymbol{v} \in \operatorname{supp} (P)}
    \lambda_{\boldsymbol{v}} \boldsymbol{v} \, : \,
    \lambda_{\boldsymbol{v}} \geq 0, \, \sum_{\boldsymbol{v} \in
  \operatorname{supp} (P)} \lambda_{\boldsymbol{v}} = 1 \right\} . \label{eq:newt:def}
\end{equation}
It is a well-known basic property of Newton polytopes that
\begin{equation}
  \operatorname{Newt} (P Q) = \operatorname{Newt} (P) + \operatorname{Newt} (Q) \label{eq:newt:prod}
\end{equation}
for all Laurent polynomials $P, Q$.

To make the statement of Theorem~\ref{thm:ct:lucas:partial:intro} more
transparent, we make the simple observation that, when considering nontrivial
sequences defined by $A (n) = \operatorname{ct} [P (\boldsymbol{x})^n]$, the Newton
polytope always contains the origin.

\begin{proposition}
  \label{prop:ct:newt0}Let $A (n) = \operatorname{ct} [P (\boldsymbol{x})^n]$ for $P
  (\boldsymbol{x}) \in \mathbb{C} [\boldsymbol{x}^{\pm 1}]$. If $A (n) \neq 0$ for
  some $n \geq 1$, then $\operatorname{Newt} (P)$ contains the origin.
\end{proposition}

\begin{proof}
  If $A (n) \neq 0$ for $n \geq 1$ then, by definition, the origin is in
  the support of $P (\boldsymbol{x})^n$ so that, in particular, $\boldsymbol{0}
  \in \operatorname{Newt} (P^n)$. Since Newton polytopes are convex by definition, it
  follows from \eqref{eq:newt:prod} that
  \begin{equation*}
    \operatorname{Newt} (P^n) = n \operatorname{Newt} (P) .
  \end{equation*}
  Thus, $\boldsymbol{0} \in n \operatorname{Newt} (P)$ and, since $n \neq 0$, we also
  have $\boldsymbol{0} \in \operatorname{Newt} (P)$.
\end{proof}

Let $\gamma > 0$. If $\boldsymbol{0} \in \operatorname{Newt} (P)$ then, corresponding to
\eqref{eq:newt:def}, the (relative) interior points of $\gamma \operatorname{Newt}
(P)$ are those that can be expressed as $\sum_{\boldsymbol{v} \in \operatorname{supp}
(P)} \lambda_{\boldsymbol{v}} \boldsymbol{v}$ with $\lambda_{\boldsymbol{v}}
\geq 0$ and $\sum_{\boldsymbol{v} \in \operatorname{supp} (P)}
\lambda_{\boldsymbol{v}} < \gamma$.

We are now in a position to prove Theorem~\ref{thm:ct:lucas:partial:intro}
which we restate as Theorem~\ref{thm:ct:lucas:partial} in a slightly more
general form, writing $\mathbb{Z}_p$ for the ring of $p$-adic integers. Our
proof is an extension of the argument in \cite{hs-lucas-x} where the special
case $M = 1$ is proved.

\begin{theorem}
  \label{thm:ct:lucas:partial}Let $P (\boldsymbol{x}) \in \mathbb{Z}_p
  [\boldsymbol{x}^{\pm 1}]$ and choose $M \geq 1$ so that $\frac{1}{M}
  \operatorname{Newt} (P)$ contains no integral (relative) interior point besides the
  origin. Then the sequence defined by $A (n) = \operatorname{ct} [P
  (\boldsymbol{x})^n]$ satisfies
  \begin{equation}
    A (p n + k) \equiv A (n) A (k) \pmod{p}
    \label{eq:ct:lucas:partial}
  \end{equation}
  for all integers $n, k \geq 0$ with $k < \frac{p}{M}$.
\end{theorem}

\begin{proof}
  Let us denote with $\Lambda_p$ the Cartier operator
  \begin{equation*}
    \Lambda_p \left[ \sum_{\boldsymbol{k} \in \mathbb{Z}^d} a_{\boldsymbol{k}}
     \boldsymbol{x}^{\boldsymbol{k}} \right] = \sum_{\boldsymbol{k} \in
     \mathbb{Z}^d} a_{p\boldsymbol{k}} \boldsymbol{x}^{\boldsymbol{k}} .
  \end{equation*}
  Let $n, k \geq 0$. Using that $P (\boldsymbol{x})^{p n} \equiv P
  (\boldsymbol{x}^p)^n$ modulo $p$ and proceeding as in \cite{hs-lucas-x} and
  \cite{bsy-constantterms}, we have
  \begin{eqnarray}
    A (p n + k) & = & \operatorname{ct} [P (\boldsymbol{x})^{p n} P (\boldsymbol{x})^k]
    \nonumber\\
    & \equiv & \operatorname{ct} [P (\boldsymbol{x}^p)^n P (\boldsymbol{x})^k] \pmod{p} \nonumber\\
    & = & \operatorname{ct} [P (\boldsymbol{x})^n \Lambda_p [P (\boldsymbol{x})^k]] . 
    \label{eq:lucas:pnk:1}
  \end{eqnarray}
  Let $c\boldsymbol{x}^{\boldsymbol{v}}$, with $c \neq 0$, be a term of $\Lambda_p
  [P (\boldsymbol{x})^k]$. This means that $c\boldsymbol{x}^{p\boldsymbol{v}}$ is a
  term of $P (\boldsymbol{x})^k$ so that, in particular,
  \begin{equation*}
    p\boldsymbol{v}= \lambda_1 \boldsymbol{v}_1 + \lambda_2 \boldsymbol{v}_2 +
     \cdots + \lambda_t \boldsymbol{v}_t
  \end{equation*}
  where $\boldsymbol{v}_i \in \operatorname{supp} (P) \subseteq \operatorname{Newt} (P)$,
  $\lambda_i \in \mathbb{Z}_{\geq 0}$ and $\lambda_1 + \lambda_2 + \cdots
  + \lambda_t = k$. Accordingly,
  \begin{equation}
    \boldsymbol{v}= \mu_1 \boldsymbol{v}_1 + \mu_2 \boldsymbol{v}_2 + \cdots + \mu_t
    \boldsymbol{v}_t, \quad \mu_i = \frac{\lambda_i}{p}, \label{eq:v}
  \end{equation}
  where $\mu_i \geq 0$ and $\mu_1 + \cdots + \mu_t = k / p <
  \frac{1}{M}$. By Proposition~\ref{prop:ct:newt0}, we may assume that
  $\boldsymbol{0} \in \operatorname{Newt} (P)$ since the congruences
  \eqref{eq:ct:lucas:partial} are trivially true otherwise. We then observe
  that, since $\boldsymbol{0} \in \operatorname{Newt} (P)$, the (relative) interior
  points of $\frac{1}{M} \operatorname{Newt} (P)$ are those that can be expressed as
  $\sum_{\boldsymbol{v} \in \operatorname{supp} (P)} \lambda_{\boldsymbol{v}}
  \boldsymbol{v}$ with $\lambda_{\boldsymbol{v}} \geq 0$ and
  $\sum_{\boldsymbol{v} \in \operatorname{supp} (P)} \lambda_{\boldsymbol{v}} <
  \frac{1}{M}$ (compare with \eqref{eq:newt:def}). Accordingly, the point
  $\boldsymbol{v}$ in \eqref{eq:v} is an integral interior point of $\frac{1}{M}
  \operatorname{Newt} (P)$. Hence $\boldsymbol{v}=\boldsymbol{0}$ since, by assumption,
  there are no other integral interior points. We conclude that
  \begin{equation*}
    \Lambda_p [P (\boldsymbol{x})^k] = \operatorname{ct} [P (\boldsymbol{x})^k]
  \end{equation*}
  which, together with \eqref{eq:lucas:pnk:1}, implies that
  \begin{equation*}
    A (p n + k) \equiv \operatorname{ct} [P (\boldsymbol{x})^n \operatorname{ct} [P
     (\boldsymbol{x})^k]] = \operatorname{ct} [P (\boldsymbol{x})^n] \operatorname{ct} [P
     (\boldsymbol{x})^k] = A (n) A (k) \pmod{p},
  \end{equation*}
  thus showing \eqref{eq:ct:lucas:partial}.
\end{proof}

We record the following simple special case of
Theorem~\ref{thm:ct:lucas:partial} that is particularly convenient to apply.
However, as illustrated in Examples~\ref{eg:apery3} and \ref{eg:delannoy3}, it
does not capture the full strength of Theorem~\ref{thm:ct:lucas:partial}.
Here, we denote with $\deg (P)$ the largest exponent to which any variable
$x_i$ or its inverse $x_i^{- 1}$ appears in the Laurent polynomial $P
(\boldsymbol{x})$.

\begin{corollary}
  \label{cor:ct:lucas:partial}Let $P (\boldsymbol{x}) \in \mathbb{Z}_p
  [\boldsymbol{x}^{\pm 1}]$. Then the sequence defined by $A (n) = \operatorname{ct} [P
  (\boldsymbol{x})^n]$ satisfies
  \begin{equation*}
    A (p n + k) \equiv A (n) A (k) \pmod{p}
  \end{equation*}
  for integers $n, k \geq 0$ with $k < \frac{p}{\deg (P)}$.
\end{corollary}

\begin{proof}
  This follows directly from Theorem~\ref{thm:ct:lucas:partial} with $M = \deg
  (P)$. Alternatively, this can be concluded from
  \cite[Lemma~3.1]{bsy-constantterms} with $r = 1$ and $Q = 1$.
\end{proof}

\begin{example}
  \label{eg:apery3}The Ap\'ery numbers $A (n)$, defined in
  \eqref{eq:apery3}, have the constant term representation
  \cite[Remark~1.4]{s-apery}
  \begin{equation*}
    A (n) = \operatorname{ct} [P (\boldsymbol{x})^n], \quad P (\boldsymbol{x}) = \frac{(x
     + y) (z + 1) (x + y + z) (y + z + 1)}{x y z} .
  \end{equation*}
  Since the Newton polytope of $P (\boldsymbol{x})$ contains no integral
  interior point besides the origin, it immediately follows from
  Theorem~\ref{thm:ct:lucas:partial} (with $M = 1$) that the Ap\'ery numbers
  $A (n)$ satisfy the Lucas congruences for all primes, as proved directly
  from the binomial sum \eqref{eq:apery3} by Gessel \cite[Theorem
  1]{gessel-super}. On the other hand, note that
  Corollary~\ref{cor:ct:lucas:partial} only provides weaker partial Lucas
  congruences because $\deg (P) = 2$.
\end{example}

\begin{remark}
  Suppose that $P, Q \in \mathbb{Z}_p [\boldsymbol{x}^{\pm 1}]$. We briefly
  indicate that Theorem~\ref{thm:ct:lucas:partial} indirectly also applies to
  constant terms of the form $A (n) = \operatorname{ct} [P (\boldsymbol{x})^n Q
  (\boldsymbol{x})]$ when $Q \neq 1$. In that case it follows, for instance,
  from \cite[Lemma~3.1]{bsy-constantterms} with $r = 1$ that
  \begin{equation}
    A (p n + k) \equiv A (k) B (n) \pmod{p} \label{eq:lucasx}
  \end{equation}
  where $B (n) = \operatorname{ct} [P (\boldsymbol{x})^n]$ provided that $p > \deg (P^k
  Q)$. The condition on $p$ is satisfied if $p > k \deg (P) + \deg (Q)$ or,
  equivalently, $k < \frac{p - \deg (Q)}{\deg (P)}$. In light of
  \eqref{eq:lucasx}, by applying Theorem~\ref{thm:ct:lucas:partial} to the
  constant term $B (n)$, we can determine the values of $A (n)$ modulo $p$
  provided that the $p$-adic digits of $n$ are suitably restricted (to satisfy
  the conditions needed for \eqref{eq:lucasx} as well the subsequent
  applications of Theorem~\ref{thm:ct:lucas:partial}).
\end{remark}

It is natural to wonder about a possible converse statement of
Theorem~\ref{thm:ct:lucas:partial}. For instance, if a constant term sequence
$(A (n))_{n \geq 0}$ satisfies partial Lucas congruences, does there
always exist a, potentially alternative, representation $A (n) = \operatorname{ct} [P
(\boldsymbol{x})^n]$ such that these congruences can be concluded from
Theorem~\ref{thm:ct:lucas:partial}? In other words:

\begin{question}
  \label{q:ct:lucas}Suppose that $A (n) = \operatorname{ct} [P_0 (\boldsymbol{x})^n Q_0
  (\boldsymbol{x})]$ with $P_0, Q_0 \in \mathbb{Z} [\boldsymbol{x}^{\pm 1}]$ is a
  constant term that satisfies $A (0) = 1$ as well the Lucas congruences
  \eqref{eq:ct:lucas:partial} for integers $n, k \geq 0$ with $k <
  \frac{p}{M}$. Does $A (n)$ necessarily have a representation $A (n) =
  \operatorname{ct} [P (\boldsymbol{x})^n]$ for some $P (\boldsymbol{x}) \in \mathbb{Z}
  [\boldsymbol{x}^{\pm 1}]$ with $\frac{1}{M} \operatorname{Newt} (P)$ containing no
  integral interior point besides the origin? If not, is this true under an
  additional natural condition?
\end{question}

We make two comments on the statement of this question before illustrating the
question and its implications in two special instances. First, we note that
the condition $A (0) = 1$ should be seen as part of satisfying the Lucas
congruences: indeed, the Lucas congruences in the alternative form
\eqref{eq:lucas:u} naturally imply $A (0) = 1$ if we apply them for $n = 0$
with an empty $p$-adic expansion so that the right-hand side of
\eqref{eq:lucas:u} is the empty product with value $1$. Moreover, the
condition $A (0) = 1$ is needed to avoid the case of the zero sequence which
can be represented as a constant term with $P = 1$ and $Q = x$ and which
trivially satisfies the congruences \eqref{eq:ct:lucas:partial} (but, clearly,
cannot have a representation of the form $A (n) = \operatorname{ct} [P
(\boldsymbol{x})^n]$).

Second, we observe that there is no loss of generality in reducing to the case
$Q_0 (\boldsymbol{x}) = 1$ in Question~\ref{q:ct:lucas}. Indeed, if $A (n)$
satisfies the Lucas congruences \eqref{eq:ct:lucas:partial} for integers $n, k
\geq 0$ with $k < \frac{p}{M}$ then, in particular,
\begin{equation*}
  A (p n) \equiv A (n) \pmod{p} .
\end{equation*}
In \cite[Proposition~5.1]{bsy-constantterms} it is shown that, for constant
terms, these congruences imply, for large enough $p$, the stronger congruences
\begin{equation}
  A (p^r n) \equiv A (p^{r - 1} n) \pmod{p^r} \label{eq:gauss}
\end{equation}
(known in the literature as Gauss congruences) and that, in fact, $A (n) = A
(0) \operatorname{ct} [P_0 (\boldsymbol{x})^n] = \operatorname{ct} [P_0 (\boldsymbol{x})^n]$ (that
is, $Q_0 (\boldsymbol{x})$ can be replaced by $1$ in the constant term
representation that we started with).

The following example illustrates Question~\ref{q:ct:lucas} for a specific
sequence. In that particular case, the answer to Question~\ref{q:ct:lucas} is
affirmative.

\begin{example}
  \label{eg:delannoy3}Consider the generalized Delannoy sequence
  \cite[\tmverbatim{A081798}]{oeis}
  \begin{equation*}
    C (n) = \sum_{k = 0}^n \binom{n}{k} \binom{n + k}{k} \binom{n + 2 k}{k},
  \end{equation*}
  which starts with the values $1, 7, 115, 2371, 54091, \ldots$ and counts
  lattice walks from the origin to $(n, n, n)$ using steps $(1, 0, 0)$, $(0,
  1, 0)$, $(0, 0, 1)$ and $(1, 1, 1)$.
  
  We can proceed as in Example~\ref{eg:u} to express $C (n)$ as the constant
  term
  \begin{equation*}
    C (n) = \operatorname{ct} \left[ \sum_{k = 0}^n \binom{n}{k} \frac{(1 + x)^{n +
     k}}{x^k} \frac{(1 + y)^{n + 2 k}}{y^k} \right] = \operatorname{ct} [P (x, y)^n]
  \end{equation*}
  with
  \begin{equation*}
    P (x, y) = (1 + x) (1 + y) \left(1 + \frac{1 + x}{x}  \frac{(1 +
     y)^2}{y} \right) .
  \end{equation*}
  The Laurent polynomial $P$ has support $\{ - 1, 0, 1 \} \times \{ - 1, 0, 1,
  2 \}$. In particular, its Newton polytope is the rectangle $[- 1, 1] \times
  [- 1, 2]$ which has the interior integral points $(0, 0)$ and $(0, 1)$.
  Hence, Theorem~\ref{thm:ct:lucas:partial:intro} can be applied with $M = 2$
  to deduce partial Lucas congruences.
  
  However, it follows from a result of McIntosh
  \cite[Theorem~6]{mcintosh-lucas} that the sequence $(C (n))_{n \geq
  0}$ satisfies the full Lucas congruences for all primes. We note that this
  can also be concluded from expressing $C (n)$ as the diagonal of the
  rational function $1 / (1 - x - y - z - x y z)$, which naturally encodes the
  lattice walk count, and applying a general result of Rowland and Yassawi
  \cite[Theorem~5.2]{ry-diag13}. It is therefore a natural special case of
  Question~\ref{q:ct:lucas} to ask whether the sequence $(C (n))_{n \geq
  0}$ has an alternative constant term representation such that
  Theorem~\ref{thm:ct:lucas:partial:intro} can be applied with $M = 1$ to
  deduce the full Lucas congruences. In this case, the answer is affirmative.
  Indeed, we can also write
  \begin{equation*}
    C (n) = \sum_{k = 0}^n \binom{n}{k} \binom{n + 2 k}{k, k, n} = \operatorname{ct}
     \left[ \sum_{k = 0}^n \binom{n}{k} \frac{(1 + x + y)^{n + 2 k}}{x^k y^k}
     \right] = \operatorname{ct} [\tilde{P} (x, y)^n]
  \end{equation*}
  with
  \begin{equation*}
    \tilde{P} (x, y) = (1 + x + y) \left(1 + \frac{(1 + x + y)^2}{x y}
     \right) .
  \end{equation*}
  The Newton polytope of $\tilde{P}$ is the triangle with vertices $(- 1, -
  1)$, $(- 1, 2)$, $(2, - 1)$. Its only interior integral point is $(0, 0)$
  and so Theorem~\ref{thm:ct:lucas:partial:intro} applies with $M = 1$ and we
  are able to conclude the full Lucas congruences.
\end{example}

\begin{example}
  The following is a specific instance for which it would be of particular
  interest to know the answer to Question~\ref{q:ct:lucas}. Consider the
  constant term sequence $(A (n))_{n \geq 0}$ defined by $A (n) =
  \operatorname{ct} [P^n]$ with the symmetric Laurent polynomial
  \begin{equation*}
    P = \frac{(zx + xy - yz - x - 1)  (xy + yz - zx - y - 1)  (yz + zx - xy -
     z - 1)}{xyz} .
  \end{equation*}
  This representation was recently discovered by Gorodetsky
  \cite{gorodetsky-ct} through a clever computational search. Gorodetsky
  succeeded in finding such constant term representations for all known
  \emph{sporadic} Ap\'ery-like sequences, the present sequence $(A (n))_{n
  \geq 0}$ typically being labeled $(\eta)$ in the literature
  \cite{asz-clausen}. For all sporadic Ap\'ery-like sequences besides
  $(\eta)$, these constant term representations are such that the Newton
  polytope of the underlying Laurent polynomial only has the origin as an
  interior integral point, thus implying (as in Example~\ref{eg:apery3} for
  the Ap\'ery numbers) that the sequence satisfies the Lucas congruences.
  The Newton polytope of the above Laurent polynomial $P$, however, contains
  the lattice points $(1, 0, 0)$, $(1, 1, 0)$ and their permutations as
  interior points so that the Lucas congruences cannot be directly concluded
  (though Theorem~\ref{thm:ct:lucas:partial} can be applied with $M = 2$ to
  imply Lucas congruences if the digits are less than $p / 2$). Nevertheless,
  Malik and the author \cite{ms-lucascongruences} had previously shown that
  all known sporadic Ap\'ery-like sequences satisfy Lucas congruences (with
  the case $(\eta)$ requiring an ad-hoc and technical argument). It is natural
  to wonder whether the sequence $(\eta)$ has a, yet-to-be-discovered,
  constant term representation $A (n) = \operatorname{ct} [\tilde{P}
  (\boldsymbol{x})^n]$ such that the Newton polytope of $\tilde{P}$ only has the
  origin as an interior integral point. An affirmative answer to
  Question~\ref{q:ct:lucas} would imply that such a representation exists. In
  any case, this example and its history underscore the difficulty in finding
  constant term representations for a given sequence. On the other hand, as
  pointed out in the introduction, constant term representations, once found,
  can be algorithmically proven (at least in principle) using, for instance,
  creative telescoping \cite{koutschan-phd}.
\end{example}

\section{More general congruences}

In addition to establishing Theorem~\ref{thm:u:lucas}, the partial Lucas
congruences for the sequence $(u (n))_{n \geq 0}$, Chamberland and
Dilcher \cite{cd-binomial09} further prove the following congruences,
covering the $p$-adic digits that are excluded in the partial Lucas
congruences.

\begin{theorem}[{\cite[Corollary~2.2]{cd-binomial09}}]
  \label{thm:u:lucas:x}Let
  \begin{equation}
    w (n) = \sum_{k = 0}^{n - 1} (- 1)^k \binom{2 n - 1}{k} \binom{n - 1}{k} .
    \label{eq:w:def}
  \end{equation}
  Then
  \begin{equation}
    u (p n + k) \equiv w (n + 1) u (k) \pmod{p}
    \label{eq:u:lucas:x}
  \end{equation}
  for all integers $n, k \geq 1$ with $(p + 1) / 2 \leq k < p$.
\end{theorem}

In \cite{cd-binomial09}, this result is obtained by working directly with
the binomial sum \eqref{eq:u:def} defining the numbers $u (n)$. In this
section, we give an alternative proof of Theorem~\ref{thm:u:lucas:x} using the
constant term representation \eqref{eq:u:ct}. This proof has the benefit of
being natural in the sense that similar computations can be performed for
other constant term sequences, though the present case offers some additional
and nongeneric simplifications. In particular, our proof makes the
relationship between the sequences $(u (n))_{n \geq 0}$ and $(w (n))_{n
\geq 0}$ transparent from the point of view of constant terms (compare
\eqref{eq:u:ct} and \eqref{eq:w:ct}). Further examples where similar
congruences are explicitly worked out for a class of constant term sequences
can be found in \cite{hs-lucas-x}.

\begin{proof}
  Recall from \eqref{eq:u:ct} that
  \begin{equation*}
    u (n) = \operatorname{ct} [P (x)^n], \quad P (x) = (1 + x)^2 \left(1 -
     \frac{1}{x} \right) .
  \end{equation*}
  We therefore have, as in \eqref{eq:lucas:pnk:1}, that, modulo $p$,
  \begin{eqnarray*}
    u (p n + k) & = & \operatorname{ct} [P (x)^{p n} P (x)^k]\\
    & \equiv & \operatorname{ct} [P (x^p)^n P (x)^k] \pmod{p}\\
    & = & \operatorname{ct} [P (x)^n \Lambda_p [P (x)^k]] .
  \end{eqnarray*}
  If $0 \leq k < p / 2$, then $\Lambda_p [P (x)^k] = \operatorname{ct} [P (x)^k]
  = u (k)$ and we obtain the partial Lucas congruences $u (p n + k) \equiv u
  (n) u (k)$ that were obtained more generally in
  Theorem~\ref{thm:ct:lucas:partial}. Here, we assume that $(p + 1) / 2
  \leq k < p$, in which case we find
  \begin{equation}
    \Lambda_p [P (x)^k] = \operatorname{ct} [P (x)^k] + x [x^p] P (x)^k,
    \label{eq:u:lucas:x:Gp}
  \end{equation}
  where $[x^m] f (x)$ denotes the coefficient of $x^m$ in the power series $f
  (x)$. Here, we used the fact that the largest exponent of $x$ among all
  terms in $P (x)^k$ is $2 k < 2 p$ while the smallest exponent is $- k > -
  p$.
  
  In general, we now need to investigate the numbers $c_p (k) = [x^p] P
  (x)^k$. In the present case, however, we have the remarkable simplification
  \begin{equation}
    [x^p] P (x)^k \equiv \operatorname{ct} [P (x)^k] \pmod{p}
    \label{eq:u:lucas:x:xp}
  \end{equation}
  for $(p + 1) / 2 \leq k < p$, which we will prove below. Assuming for
  now the truth of \eqref{eq:u:lucas:x:xp} and combining it with
  \eqref{eq:u:lucas:x:Gp}, we obtain that
  \begin{equation*}
    \Lambda_p [P (x)^k] \equiv (1 + x) \operatorname{ct} [P (x)^k] = (1 + x) u (k)
     \pmod{p} .
  \end{equation*}
  In particular, we conclude that
  \begin{equation*}
    u (p n + k) \equiv \operatorname{ct} [P (x)^n \Lambda_p [P (x)^k]] \equiv
     \operatorname{ct} [P (x)^n (1 + x)] \cdot u (k) = v (n) u (k) \pmod{p},
  \end{equation*}
  which proves generalized Lucas congruences involving the auxiliary sequence
  defined by
  \begin{equation*}
    v (n) = \operatorname{ct} [P (x)^n (1 + x)] .
  \end{equation*}
  Indeed, this constant term sequence is precisely the sequence $(w (n +
  1))_{n \geq 0}$ introduced by Chamberland and Dilcher as the binomial
  sum \eqref{eq:w:def}. This can be seen from the following computation
  \begin{eqnarray}
    w (n + 1) & = & \sum_{k = 0}^n (- 1)^k \binom{2 n + 1}{k} \binom{n}{k}
    \nonumber\\
    & = & \operatorname{ct} \left[ \sum_{k = 0}^n (- 1)^k \binom{n}{k} \frac{(1 +
    x)^{2 n + 1}}{x^k} \right] \nonumber\\
    & = & \operatorname{ct} \left[ (1 + x)^{2 n + 1} \left(1 - \frac{1}{x} \right)^n
    \right] = \operatorname{ct} [P (x)^n (1 + x)],  \label{eq:w:ct}
  \end{eqnarray}
  which is along the same lines as the earlier derivation of \eqref{eq:u:ct}.
  
  It remains to show that \eqref{eq:u:lucas:x:xp} is true for $(p + 1) / 2
  \leq k < p$. Since $[x^p] P (x)^k = \operatorname{ct} [x^{- p} P (x)^k]$, the
  congruence \eqref{eq:u:lucas:x:xp} is equivalent to
  \begin{equation}
    \operatorname{ct} [(1 - x^{- p}) P (x)^k] \equiv 0 \pmod{p} .
    \label{eq:u:lucas:x:xp:diff}
  \end{equation}
  Since $P (x) = (1 + x)^2 \left(1 - \frac{1}{x} \right)$, we have
  \begin{equation*}
    \operatorname{ct} [(1 - x^{- p}) P (x)^k] = \operatorname{ct} \left[ \left(1 -
     \frac{1}{x^p} \right) (1 + x)^p (1 + x)^{2 k - p} \left(1 - \frac{1}{x}
     \right)^k \right] .
  \end{equation*}
  Using $(1 + x)^p \equiv (1 + x^p)$ modulo $p$ and $(1 - x^{- p}) (1 + x^p) =
  (x^p - x^{- p})$, we conclude
  \begin{equation*}
    \operatorname{ct} [(1 - x^{- p}) P (x)^k] \equiv \operatorname{ct} \left[ \left(x^p -
     \frac{1}{x^p} \right) (1 + x)^{2 k - p} \left(1 - \frac{1}{x} \right)^k
     \right] \pmod{p} .
  \end{equation*}
  The right-hand side is now seen to be $0$ because $(1 + x)^{2 k - p}$ is a
  polynomial (here, we use that $k \geq (p + 1) / 2$) in $x$ of degree at
  most $2 k - p < p$ while $(1 - x^{- 1})^k$ is a polynomial in $x^{- 1}$ of
  degree at most $k < p$ (and so the product does not have a term that is a
  multiple of either $x^p$ or $x^{- p}$). This proves
  \eqref{eq:u:lucas:x:xp:diff}.
\end{proof}

\section{Conclusions}

We have shown that all sequences representable as the constant terms of the
powers of a multivariate Laurent polynomial satisfy Lucas congruences modulo
all primes with a restriction on the allowed digits. In a related but somewhat
different direction, Rowland \cite{rowland-lucas-p2} considers the classical
Lucas congruences \eqref{eq:binom:lucas} for the binomial coefficients and
characterizes those digits $(r, s)$ in base $p$ such that the congruences
\begin{equation*}
  \binom{p n + r}{p k + s} \equiv \binom{n}{k} \binom{r}{s} \pmod{p^2}
\end{equation*}
hold for all integers $n, k \geq 0$. Rowland, Yassawi and Krattenthaler
\cite{ryk-lucas-p2} further consider the question of partial Lucas
congruences modulo $p^2$ for the Ap\'ery numbers \eqref{eq:apery3}. In
general, such congruences modulo higher powers of $p$ are more isolated than
the general congruences in Theorem~\ref{thm:ct:lucas:partial:intro}.

A motivation of the works of Chamberland and Dilcher
\cite{cd-binomial06,cd-binomial09} was the fact that the sequences $(u_{a,
b}^{\varepsilon} (n))_{n \geq 0}$ defined in \eqref{eq:uabe:def} are
connected to Wolstenholme's theorem. Indeed, they show that
\begin{equation*}
  u_{a, b}^{\varepsilon} (p) \equiv 1 + (- 1)^{\varepsilon} 2^b \pmod{p^3}
\end{equation*}
for any prime $p \geq 5$, except when $(\varepsilon, a, b) = (0, 0, 1)$
or $(0, 1, 0)$. It remains an open problem to understand the composite numbers
$n$ satisfying this congruence as well. Andrews \cite{andrews-qcong99}
obtained a $q$-analog of Wolstenholme's congruences for the binomial
coefficients. In this spirit, it could be of interest to develop a $q$-version
of the present results.

The sequences $(u (n))_{n \geq 0}$ and $(u_{a, b}^{\varepsilon} (n))_{n
\geq 0}$ studied by Chamberland and Dilcher \cite{cd-binomial09} were
expressed as constant terms in \eqref{eq:u:ct} and \eqref{eq:uabe:ct}, and
these representations were deduced in a systematic manner from the defining
binomial sums. In general, however, it is an open problem to determine whether
a given integer sequence can be represented as a constant term $\operatorname{ct} [P
(\boldsymbol{x})^n Q (\boldsymbol{x})]$ for Laurent polynomials $P (\boldsymbol{x}),
Q (\boldsymbol{x}) \in \mathbb{Z} [\boldsymbol{x}^{\pm 1}]$. This question was
raised by Zagier \cite[p.~769, Question~2]{zagier-de} and Gorodetsky
\cite{gorodetsky-ct} in the case $Q = 1$ and appears in \cite{s-schemes}
for general $Q$. A classification in the $C$-finite case was completed in
\cite{bsy-constantterms} but the general case remains open, including the
special cases of hypergeometric sequences or those with an algebraic
generating function (or, even, the intersection of those).

\end{document}